\documentclass[reqno,12pt]{amsart}
\usepackage{a4wide}
\usepackage{amsmath}
\usepackage{amssymb}
\usepackage{amsthm}

\usepackage[utf8]{inputenc} 
\usepackage[normalem]{ulem}

\numberwithin{equation}{section}

\newtheorem{theo}{Theorem}

\newtheorem{coro}{Corollary}
\newtheorem{prop}{Proposition}
\newtheorem{lem}{Lemma}

\theoremstyle{remark}

\def\Punkt(#1 #2){\move(#1 #2)\fcir f:0 r:.10}
\def\DickPunkt(#1 #2){\move(#1 #2)\fcir f:0 r:.15}

\def\og{\leavevmode\raise0.25ex
                       \hbox{$\scriptscriptstyle\ll$}}
\def\fg{\leavevmode\raise0.25ex
                       \hbox{$\scriptscriptstyle\gg$}}
\ProvideTextCommandDefault{\guillemotleft}%
  {\leavevmode\hbox{\fontencoding{U}\fontfamily{lasy}%
                    \fontseries{m}\fontshape{n}\selectfont
    (\kern-0.20em(\kern+0.20em}\nobreak}
\ProvideTextCommandDefault{\guillemotright}%
  {\nobreak\leavevmode
   \hbox{\fontencoding{U}\fontfamily{lasy}
         \fontseries{m}\fontshape{n}\selectfont
   \kern+0.20em)\kern-0.20em)}}

\def\al{\alpha}

\def\si{\sigma}
\def\ta{\tau}

\def\ph{\varphi}

\def\({\left(}
\def\){\right)}
\def\[{\left[}
\def\]{\right]}

\def\fl#1{\left\lfloor#1\right\rfloor}
\def\cl#1{\left\lceil#1\right\rceil
}

\def\lcm{\operatorname{lcm}}

\def\Qbar{\overline{\mathbb Q}}

\begin{document}

\title[]{Arithmetic properties of the Taylor coefficients of differentially algebraic power
series}

\author[]
       {C. Krattenthaler$^\dagger$ and T. Rivoal}
\date{\today}

\address{C. Krattenthaler, Fakult\"at f\"ur Mathematik, Universit\"at Wien,
Oskar-Morgenstern-Platz~1, A-1090 Vienna, Austria.
WWW: \tt http://www.mat.univie.ac.at/\~{}kratt.}

\address{T. Rivoal,
Institut Fourier, Universit{\'e} Grenoble Alpes, CNRS,
CS 40700, 38058 Grenoble cedex 9,
France.
WWW: \tt https://rivoal.perso.math.cnrs.fr.}

\thanks{$^\dagger$Research partially supported by the Austrian
Science Foundation FWF, grant 10.55776/F1002, in the framework
of the Special Research Program ``Discrete Random Structures:
Enumeration and Scaling Limits". 
}

\subjclass[2020]{Primary 11B37;
Secondary 05A10 11B65 30B10}

\keywords{Non-linear recurrences, differentially algebraic power
series, Gregory coefficients, Bernoulli numbers, Weierstra{\ss} $\wp$ function,
Painlev\'e equations, elliptic modular invariant, Kepler's equation,
dilogarithm}

\begin{abstract} Let $f=\sum_{n=0}^\infty f_n x^n \in \Qbar[[x]$ be a solution of an algebraic differential equation $Q(x,y(x), \ldots, y^{(k)}(x))=0$, where $Q$ is a multivariate polynomial with coefficients in~$\Qbar$.  
The sequence $(f_n)_{n\ge 0}$ satisfies a non-linear recurrence, whose expression involves a polynomial $M$ of degree $s$. When the equation is linear, $M$ is its indicial polynomial at the origin. We show  that when $M$ is split over $\mathbb Q$, there exist two positive integers $\delta$ and $\nu$ such that 
the denominator of $f_n$ divides $\delta^{n+1}(\nu n+\nu)!^{2s}$ for all $n\ge 0$, generalizing a well-known property when the equation is linear. This proves in this case a strong form of a conjecture of Mahler that P\'olya--Popken's upper bound $n^{\mathcal{O}(n\log(n))}$ for the denominator of $f_n$ is not optimal.
This also enables us to make Sibuya and Sperber's bound $\vert f_n\vert_v\le e^{\mathcal{O}(n)}$, for all finite places $v$ of~$\Qbar$, explicit in this case. Our method is completely effective and rests upon a detailed $p$-adic analysis of the above mentioned non-linear recurrences. Finally, we present various examples of differentially algebraic functions for which the associated polynomial $M$ is split over $\mathbb Q$, among which are Weierstra{\ss}' elliptic $\wp$ function, solutions of Painlev\'e equations, and Lagrange's solution to Kepler's equation.
\end{abstract}

\maketitle

\begin{flushright}
{\em In the memory of Bernard Malgrange}
\end{flushright}

\section{Introduction}

We denote by $\Qbar$ the field of algebraic numbers and by $\mathcal{O}_{\Qbar}$ the ring of algebraic integers.  
Let us consider a non-trivial algebraic differential equation 
\begin{equation} \label{eq:diffalgeq}
Q(x,y(x), \ldots, y^{(k)}(x))=0,
\end{equation}
where $Q\in \mathbb C[X,Y_0, \ldots, Y_{k}]$; here and in the sequel, by ``non-trivial'' we mean that $Q$ is of degree $\ge 1$ in at least one of the indeterminates $Y_j$. Let us assume that \eqref{eq:diffalgeq} has a power series solution $f(x):=\sum_{n=0}^\infty f_n x^n\in \mathbb C[[x]]$. We say that $f$ is a {\em differentially algebraic} power series, and write DA in short. For instance, the Taylor expansions at $x=0$ of $x/\log(1+x)$ and of $\tan(x)$ are DA power series, as these functions are solutions of the equations $x(1+x)y'+y^2-(1+x)y=0$ and $y'-y^2-1=0$, respectively, which are both Riccati equations.

A natural question arises: is it possible to bound  $\vert f_n \vert$ non-trivially from above and from below? In the Archimedean case, an upper bound was first obtained by Maillet~\cite{maillet}, and made more precise by various authors, in particular by Malgrange~\cite{malgrange}, and in the linear case by Perron~\cite{perron}. It is recalled in~\cite[p.~200]{mahler}
in the following form: there exist $\alpha, \beta>0$ that depend on $Q$ and $f$ such that $\vert f_n\vert \le \alpha n!^{\beta}$ for all $n\ge 0$. 
A lower bound was obtained by Popken~\cite{popken} under the additional assumption that $f_n\in \Qbar$ for all $n\ge 0$, when $\Qbar$ is embedded into $\mathbb C$: there exist $\alpha, \beta>0$ that depend on $Q$ and $f$ such that either $f_n=0$ or $\vert f_n\vert \ge \alpha n^{-\beta n\log(n)}$ for all $n\ge 1$, and  if $Q$ is a linear polynomial in $Y_0, \ldots, Y_k$ ({\em i.e.}, the differential equation is linear inhomogeneous with coefficients in $\mathbb C[X]$), then either $f_n=0$ or $\vert f_n\vert \ge \alpha n^{-
\beta n}$ for all $n\ge 1$. Similar lower bounds had been obtained earlier by P\'olya~\cite{polya} when $f_n\in \mathbb Q$.

Popken's lower bound has applications in the Theory of Transcendental Numbers, in particular to the values of Weierstra{\ss}' $\wp$ function; see \cite[pp.~207--212]{mahler}. It was reproved in detail by Mahler in~\cite[Chapter~8]{mahler}, following the method used by P\'olya~\cite {polya} in the rational case. There it is shown to be a simple consequence of the following result. Let $d_n$ denote the least positive integer such that $d_n f_n$ is in $\mathcal{O}_{\Qbar}$. We shall from now on say that $d_n$ is ``the denominator of'' $f_n$. Then there exist $\alpha, \beta>0$ that depend on $Q$ and $f$ such that 
\begin{equation}\label{eq:bounddnmahler}
d_n \le \alpha n^{\beta n\log(n)}, \quad n\ge 1,
\end{equation}
and if $Q$ is a linear polynomial in $Y_0, \ldots, Y_k$ then $d_n \le \alpha n^{\beta n}$  for all $n\ge 1$. Popken's theorem follows by combining Maillet's upper bound for the Galois conjugates $\sigma(f_n)$ of $f_n$ (because $\sum_n \sigma(f_n)x^n$ is also DA) and the upper bound \eqref{eq:bounddnmahler} for $d_n$, and then by considering the norm  of $d_nf_n$ over $\mathbb Q$, which is a positive integer when $f_n\neq 0$.

\medskip

Moreover, Mahler \cite[p.~212]{mahler} conjectured in rather~loose terms that the exponent $\beta n\log(n)$ in Popken's lower bound  {\em ``can probably be improved to something like''}\break ${\beta}n\log\log(n)$. This would make the general bound closer to the bound in the linear case. Such an improvement, which could be difficult to detect numerically, would follow immediately from the same improvement on the exponent of $n!$ in~\eqref{eq:bounddnmahler} because in Maillet's upper bound the exponent is even $\mathcal{O}(1)$. A different type of improvement of \eqref{eq:bounddnmahler} was made by Sibuya and Sperber in \cite[p.~112, Eq.~(III')]{SiSpAB}, who proved that, for all finite places~$v$ of~$\Qbar$, $\vert f_n\vert_{v} \le e^{c n}$ for some non-explicit constant $c>0$ that depends on $f$, $Q$, and~$v$; this had been conjectured by Dwork. This proves in particular that $f$ has positive $v$-adic radius of convergence in $\mathbb C_p$. Earlier, Mahler~\cite{mahlervadic} had proved a lower bound for $\vert f\vert_v$ similar to Popken's in the Archimedean case. We refer to \cite{SiSpAA} for a survey of Archimedean and non-Archimedean estimates for $f_n$ when $f$ is a DA power series, in particular for (possibly not optimal) explicit values  of the constants $\alpha$ and $\beta$.

Mahler's denominator conjecture was the starting point of our investigations, and this also led us naturally to study Sibuya and
Sperber's $v$-adic bounds.  
Let us first consider two simple examples in the linear case, both of hypergeometric type. 
The series $h_1(x):=\sum_{n=0}^{\infty} \frac{x^k}{\prod_{k=0}^{n-1} (k^2+1)}$ is a DA power series solution to 
$x^2y'''+xy''+(1-x)y'-y=0$, 
and the associated denominator $d_n$ satisfies
$d_n=
\prod_{k=0}^{n-1} (k^2+1) \le \prod_{k=0}^{n-1} (k+1)^2 = n!^2 
$,
as expected. Notice however that, conjecturally, there exist no integers $\delta,\nu, \mu, s\ge 0$ such that $d_n$ divides $\delta^{n+1}(\nu n+\mu)!^s$ for all $n\ge 0$~(\footnote{Assume on the contrary the existence of such integers. Then, any prime divisor $p$ of $d_n$ divides either $\delta$ or $(\nu n+\mu)!$, hence  $p\le \max(\delta, \nu n+\mu)$. But it is widely believed, because it is an instance of Bouniakovsky's conjecture, that there exist infinitely many primes of the form $m^2+1$ where $m\in \mathbb N$, so that we would have $(n-1)^2+1\le \max(\delta, \nu n+\mu)$ for infinitely many $n$. This is not possible. See \cite{sah} and references therein for estimates of the $p$-adic valuation of products of the form $\prod_{1\le k\le n} \vert q(k)\vert$, where $q$ is a polynomial of degree $\ge 2$, irreducible over $\mathbb Q$, and taking integer values at integer arguments.}). On the other hand, the series $h_2(x):=\sum_{n=0}^{\infty} \frac{x^k}{\prod_{k=0}^{n-1} (2k+1)}$ is also a DA power series solution to $2xy''+(1-x)y'-y=0$. Here, the associated denominator $d_n=\prod_{k=0}^{n-1} (2k+1)$ is not only bounded by 
$
\prod_{k=0}^{2n-1} (k+1) = (2n)! 
$,
but this time it {\em divides} 
$(2n)!$.  This is not surprising: it is well-known that such a divisibility phenomenon holds for all power series in~$\Qbar[[x]]$ that are solution of a homogeneous linear differential equation with coefficients in~$\Qbar[x]$ and whose indicial polynomial at the origin has only {\em rational\/} roots. This is the case for $h_2$ (roots 0 and $1/2$) but not for $h_1$ (roots 0 and $1\pm i$). See \S\ref{ssec:linearcase} for the details.

A little more surprising may be the fact that a divisibility phenomenon also numerically holds for certain DA series that are not necessarily solutions of a linear differential equation. Let us  consider for instance the series
$f(x)=1+\sum_{n=1}^\infty f_nx^n$ that satisfies the Riccati equation $xf'(x)-xf(x)^2+af(x)-a=0$, where $a\ge 1$ is a fixed integer. The sequence $(f_n)_{n\ge 0}$ satisfies the non-linear recurrence relation
\begin{equation}\label{eq:fna}
f_{n+1} =\frac{1}{n+a+1} \sum_{j=0}^n f_jf_{n-j}, \; \quad n\ge 0, \; f_0:=1,
\end{equation}
and we observed numerically that seemingly $n!\,(n+a)!\,f_n \in \mathbb Z$ for all $n\ge 0$. (Note that, if $a=0$, then $f_n=1$ for all $n\ge 0$.) Our first thought was that  this observation would be easy to prove by considering the sequence $\varphi_n:=n!\,(n+a)!\,f_n$, which is solution of the recurrence
$$
\varphi_{n+1}=\sum_{j=0}^n \frac{n+1}{(j+1)_{a}}\binom{n}{j}\binom{n+a}{j}\varphi_j\varphi_{n-j},  \;\quad n\ge 0, \;\varphi_0:=a!,
$$
where the {\it Pochhammer symbol\/} $(\alpha)_m$ is defined by $\al(\al+1)\cdots(\al+m-1)$ for $m\ge1$, and $(\al)_0=1$.
The binomial coefficients are obviously an important gain but this transformation comes with a new ``big'' denominator $(j+1)_{a}$. By a careful $p$-adic analysis of the summand (see Proposition~\ref{prop:fastallg1} in \S\ref{sec:prooftheomain}), we shall prove that $\delta^{n+1}\varphi_n\in \mathbb Z$ for all $n\ge 0$ for some integer $\delta\ge 1$. The proof of this arithmetic property, which is weaker than what seems to be true, is much more complicated than what one might expect at first sight. On the other hand, this yields a better result for $f_n$ than the upper bound \eqref{eq:bounddnmahler}, {\em i.e.}, we have $\delta^{n+1}n!\,(n+a)!\,f_n \in \mathbb Z$ for all $n\ge 0$, and it even proves Mahler's conjecture in this case in a stronger form. This also makes the upper bound for $\vert f\vert_v$ by Sibuya and Sperber explicit.

\medskip

To explain our contribution towards Mahler's denominator conjecture, we start by reviewing properties proved by Mahler in~\cite[Chapter~8]{mahler}. We assume that the differential equation~\eqref{eq:diffalgeq} has a solution 
$f(x):=\sum_{n=0}^\infty f_n x^n \in \Qbar[[x]]$. Then, by~\cite[p.~186, \S120]{mahler}, the differential equation can be assumed to have coefficients in~$\Qbar$ without loss of generality.
Mahler proved in~\cite[p.~194]{mahler} that the sequence of Taylor coefficients $(f_n)_{n\ge 0}$ satisfies a non-linear recurrence 
\begin{equation} \label{eq:Rek-Polyabis} 
f_{n+1}=\frac {1} {M(n)}
\sum_{\si=\si_1}^{\si_2}
\sum_{k=1}^{k_0}
\underset{0\le j_1,\dots,j_k\le n}{\sum_{j_1+\dots+j_k=n-\sigma}}
P_{\si,k}(n,j_1,j_2,\dots,j_k)f_{j_1}f_{j_2}\cdots
f_{j_k},\quad  \text{for }n\ge N,
\end{equation}
where  $N$ is some non-negative integer, 
$M(X)\in \mathcal{O}_{\Qbar}[X]$ 
vanishes for no $n\ge N$, the coefficients $P_{\si,k}(n,j_1,j_2,\dots,j_k)$ are in $\mathcal{O}_{\Qbar}$, the constants $\si_1$, $\si_2$ are integers, $s$ and $k_0$ are positive integers. More precisely, the $P_{\si,k}(n,j_1,j_2 \ldots, j_k)$ are piecewise polynomials in $n, j_1, j_2, \ldots, j_k$ with coefficients in $\mathcal{O}_{\Qbar}$. 
The initial values $f_0,f_1,\dots,f_N$ are algebraic numbers with common
denominator~$D$. It should be observed that the restriction
$j_1,\dots,j_k\le n$ can be safely ignored for {\it non-negative}~$\si$,
while it has an effect for {\it negative}~$\si$.  
The process to determine explicitly the recurrence 
\eqref{eq:Rek-Polyabis} from the polynomial $Q$ can be complicated in general but it is effective; it is a task that can be performed easily for simple examples. Following P\'olya in the rational case, Mahler then proved the upper bound on the denominator of~$f_n$ as follows: he showed by induction on $n$ that $\Delta_n:=\prod_{j=N}^n \vert A(j)\vert^{\lfloor \frac{(m-1)n+1}{(m-1)j+1}\rfloor}$ is a denominator of $f_n$ for $n\ge N$, where $A(X)\in \mathbb Z[x]$ vanishes for no integer $n\ge N$ and is constructed from $M$ in \cite[p.~202]{mahler}, and where the integer $m\ge 1$, which depends only on $Q$ and is defined in \cite[p.~187]{mahler}, is equal to 1 if and only if $Q$ is linear in $Y_0, \ldots, Y_k$. 
The denominator $d_n$ of~$f_n$ divides $\Delta_n$, which then has to be bounded. When $m=1$, the exponent of each $A(j)$ in the product for $\Delta_n$ is always 1 and thus independent of~$n$, while if $m\ge 2$ it depends on~$n$. This explains the difference between the upper bounds for $d_n$ in the linear and non-linear cases.

\medskip

In this paper, we shall consider the particular situation where {\em the polynomial $M$ in~\eqref{eq:Rek-Polyabis} is split over $\mathbb Q$}, {\em i.e.}, the sequence of Taylor coefficients $(f_n)_{n\ge 0}$ satisfies a non-linear recurrence of the form
\begin{multline} \label{eq:Rek-Polya} 
f_{n+1}=
\\
\frac {1} {C
  \prod _{i=1} ^{s}(a_in+b_i)}
\sum_{\si=\si_1}^{\si_2}
\sum_{k=1}^{k_0}
\underset{0\le j_1,\dots,j_k\le n}{\sum_{j_1+\dots+j_k=n-\sigma}}
P_{\si,k}(n,j_1,j_2,\dots,j_k)f_{j_1}f_{j_2}\cdots
f_{j_k},\quad  \text{for } n\ge N,
\end{multline}
where $N$
is some non-negative integer,  the coefficients $P_{\si, k}(n,j_1,j_2,\dots,j_k)$ are in $\mathcal{O}_{\Qbar}$~(\footnote{Due to the normalization adopted in~\eqref{eq:Rek-Polya}, the coefficients $P_{\si,k}$ in~\eqref{eq:Rek-Polya} might differ from the $P_{\si,k}$ in~\eqref{eq:Rek-Polyabis} by a common non-zero algebraic factor (independent of $\si$ and $k$).  The important point for us is that we can still assume without loss of generality that they are in $\mathcal{O}_{\Qbar}$; indeed, this can always be achieved by changing $C$ to $mC$ for a suitable integer $m\ge 2$ if necessary.}), 
$C$~
is a non-zero integer,  $\si_1$, $\si_2$ are integers, $s$ and $k_0$ are positive integers, 
the $b_i$'s are integers, the~$a_i$'s
are positive integers such that $\gcd(a_i,b_i)=1$ and $a_in+b_i\neq 0$ for all~$i\in \{1, \ldots, s\}$ and all $n\ge N$. 
The initial values $f_0,f_1,\dots,f_N$ are algebraic numbers with common
denominator~$D$.  
In this situation, our main result is a proof of Mahler's denominator conjecture in a much stronger form, {\em i.e.}, a divisibility property of $d_n$ which in particular implies that $\log(n)$ can be replaced by $\mathcal{O}(1)$, not just $\log\log(n)$.
\begin{theo}\label{theo:main2}
Let  $f\in \Qbar[[x]$ be a solution of a non-trivial algebraic differential equation $Q(x,y(x), \ldots, y^{(k)}(x))=0$, where $Q\in \Qbar[X,Y_0, \ldots, Y_{k}]$.
More specifically, we assume that the sequence $(f_n)_{n\ge 0}$ of Taylor coefficients of $f$ satisfies a non-linear recurrence of the form~\eqref{eq:Rek-Polya}. 

Then there exist two positive integers $\delta$ and $\nu$ such that 
the denominator of $f_n$ divides $\delta^{n+1}(\nu n+\nu)!^{2s}$ for all $n\ge 0$,
where $s$ is the degree of the denominator polynomial on the
right-hand side of~\eqref{eq:Rek-Polya}.
\end{theo}
As a consequence, we obtain an explicit version of the $v$-adic upper bound of Sibuya and Sperber in this particular situation, which seems to be new.
\begin{coro} \label{coro:1} In the setting of Theorem \ref{theo:main2}, 
for all finite places $v$ of~$\Qbar$ over any given rational prime number~$p$, we have 
\begin{equation}\label{eq:vadicbound}
\vert f_n\vert_v \le p^{\left(v_p(\delta)+\frac{2s\nu}{p-1}\right)(n+1)}, \quad n\ge 0,
\end{equation}
where we take the standard normalization $\vert p\vert_v:=1/p$.
\end{coro}
(On the right-hand side of \eqref{eq:vadicbound}, $v_p$ denotes the usual $p$-adic valuation, and the exponent involves $\nu$ and not the place $v$.)
 The proof of Theorem \ref{theo:main2} provides completely explicit, albeit complicated, formulas for $\delta$ and $\nu$ in terms of $C,D$, $k_0$, $a_1,a_2,\dots,a_s$ and $b_1,b_2,\dots,b_s$. When $\sigma_1\ge 0$, these formulas, in a refined form, are given by Theorem~\ref{thm:main} in \S\ref{sec:prooftheomain}; they are not always sharp. When $\sigma_1<0$ for a sequence $(f_n)_{n\ge 0}$ defined for all $n\ge N$ by a recurrence like~\eqref{eq:Rek-Polya}, we shall prove in a constructive way (Lemma~\ref{lem:neg->pos}) that we can always find another recurrence for $(f_n)_{n\ge 0}$ like~\eqref{eq:Rek-Polya}  in which now $\sigma_1\ge 0$, but for $n\ge \widetilde{N}$ which may differ from~$N$. We shall explain in \S\ref{ssec:lemma1} that Lemma \ref{lem:neg->pos} is a variant of \cite[p.~382, Lemma~2.2]{SiSpAA}. The latter also provides in principle another way of finding a recurrence of type \eqref{eq:Rek-Polya} from a recurrence of type \eqref{eq:Rek-Polyabis}, even though it seems not to be entirely constructive, as we explain at the end of \S\ref{ssec:lemma1}.

 The arithmetic assumption on the polynomial $M$ implies in particular the rationality of all the roots (known as the local exponents) of the indicial polynomial at 0 when $Q$ is  linear in $Y_0, \ldots, Y_k$. The procedure in \cite[p.~382]{SiSpAA} has the theoretical interest of explaining why $M$ can be
viewed as a natural generalization to the non-linear case of the indicial polynomial at the origin in the linear case. A similar assumption ap\-pears in~\cite[p.~400, $(ii)$]{SiSpAA}, where Popken's lower bound is then improved to $\vert f_n\vert\ge \alpha n^{-\beta n}$ for some non-explicit positive constants $\alpha$ and $\beta$. It does not seem that the stronger and more precise divisibility properties in  Theorem~\ref{theo:main2} and Theorem \ref{thm:main} below are also proven. It turns out that the arithmetic assumption holds in many interesting situations.  
We shall elaborate on these examples in \S\ref{sec:examples}.

\medskip

In the next section, we present our proof of Theorem~\ref{theo:main2}.
It consists in several steps. In Lemma~\ref{lem:neg->pos} we show that, 
if in the sum on the right-hand side of~\eqref{eq:Rek-Polya} negative~$\si$'s
contribute, then one can find an equivalent recurrence of the same type
without negative~$\si$'s contributing; we also compare it with Lemma 2.2 in \cite{SiSpAA}. We then restate Theorem~\ref{theo:main2}
in a more precise, but also {\it more general\/} form, see Theorem~\ref{thm:main}.
The key arithmetic part of its proof is taken care of separately in
Proposition~\ref{prop:fastallg1}. In the last section, Section~\ref{sec:examples},
we discuss many examples in which our theorems apply, amongst which are
Gregory coefficients, tangent numbers, Bernoulli numbers, Taylor coefficients of the
Weierstra{\ss} elliptic function, of solutions to
Painlev\'e equations, of modular forms, of solutions to Kepler's equation,
of the compositional inverse of the dilogarithm. We conclude the paper with a discussion of Mahler's conjecture on a simple example of a recurrence for which $M$ is non-split over $\mathbb Q$.

\section{Proof of Theorem \ref{theo:main2}} \label{sec:prooftheomain}

\subsection{An effective procedure to remove negative $\sigma$'s}\label{ssec:lemma1}

In this subsection, we show that, if there are negative $\si$'s contributing
to the right-hand side of~\eqref{eq:Rek-Polya}, then there is an
equivalent recurrence with all the characteristics of~\eqref{eq:Rek-Polya} but which avoids negative~$\si$'s, {\em i.e.}, it satisfies $\sigma_1\ge 0$. By equivalent recurrences of type~\eqref{eq:Rek-Polya}, we mean that the given sequence $(f_n)_{n\ge 0}$ satisfies both recurrences for all $n$ large enough, but ``large enough'' is different in each recurrence. At the end of this subsection, we shall also mention a non-constructive variant of Lemma \ref{lem:neg->pos}, due to Sibuya and Sperber, which nonetheless has important theoretical consequences for us.

\begin{lem} \label{lem:neg->pos}
For any recurrence of the form~\eqref{eq:Rek-Polya}
in which negative $\si$'s appear,
there is an equivalent one with only non-negative $\si$'s.
\end{lem}
The proof of Lemma~\ref{lem:neg->pos} is constructive, and free of any indeterminacy. As an inspection of the proof reveals, the conclusion of Lemma~\ref{lem:neg->pos} holds more generally also for a recurrence of type~\eqref{eq:Rek-Polyabis}.
\begin{proof}
We consider a recurrence \eqref{eq:Rek-Polya} in which $\si_1$ is negative.
Let $\bar N:=\max\{N,\si_2-k_1\si_1\}$. (The background for this choice will
become clear later on.)
We assume the $f_j$'s  with $j<\bar N$ as initial conditions,
meaning that $f_j$ with $N<j<\bar N$ has been computed via
the recurrence~\eqref{eq:Rek-Polya}.

Now, in the sums over~$j_t$'s,
we isolate terms containing $f_j$'s with $j<-\si_1$
(recall  that $-\si_1>0$) in
separate sums. This converts~\eqref{eq:Rek-Polya} into
\begin{align*} 
f&_{n+1}=\frac {1} {C
  \prod _{i=1} ^{s}(a_in+b_i)}
\sum_{\si=\si_1}^{\si_2}
\sum_{k=1}^{k_0}
\sum_{\ell=0}^k\binom k\ell
\sum_{0\le j_1,\dots,j_\ell<-\si_1}
f_{j_1}\cdots f_{j_\ell}\\
&\kern2cm
\cdot
\underset{-\si_1\le j_{\ell+1},\dots,j_k\le n}{\sum_{j_{\ell+1}+\dots+j_k
    =n-\sigma-j_1-\dots-j_\ell}}
P_{\si, k}(n,j_1,j_2,\dots,j_k)f_{j_{\ell+1}}\cdots f_{j_k}
\\
\end{align*}
\begin{align*}
&=\frac {1} {C
  \prod _{i=1} ^{s}(a_in+b_i)}
\sum_{\si=\si_1}^{\si_2}
\sum_{k=1}^{k_0}
\sum_{\ell=0}^k\binom k\ell
\sum_{0\le j_1,\dots,j_\ell<-\si_1}
f_{j_1}\cdots f_{j_\ell}
\\
&\kern10pt
\cdot
\underset{0\le j_{\ell+1},\dots,j_k\le n+\si_1}{\sum_{j_{\ell+1}+\dots+j_k
    =n-\sigma-j_1-\dots-j_\ell+(k-\ell)\si_1}}
P_{\si, k}(n,j_1,\dots,j_\ell,j_{\ell+1}-\si_1,
\dots,j_k-\si_1)f_{j_{\ell+1}-\si_1}\cdots f_{j_k-\si_1}.
\end{align*}
The binomial coefficient arises since the $\ell$ terms $f_j$ with
$j<-\si_1$ might appear at any of the $k$~possible positions.

\medskip

Now we write $g_m=f_{m-\si_1}$. In this notation, the above recurrence
reads
\begin{multline*} 
g_{n+1+\si_1}
=\frac {1} {C
  \prod _{i=1} ^{s}(a_in+b_i)}
\sum_{\si=\si_1}^{\si_2}
\sum_{k=1}^{k_0}
\sum_{\ell=0}^k\binom k\ell
\sum_{0\le j_1,\dots,j_\ell<-\si_1}
f_{j_1}\cdots f_{j_\ell}
\\
\cdot
\underset{0\le j_{\ell+1},\dots,j_k\le n+\si_1}{\sum_{j_{\ell+1}+\dots+j_k
    =n-\sigma-j_1-\dots-j_\ell+(k-\ell)\si_1}}
P_{\si, k}(n,j_1,\dots,j_\ell,j_{\ell+1}-\si_1,
\dots,j_k-\si_1)g_{j_{\ell+1}}\cdots g_{j_k},
\end{multline*}
or, after performing the shift $n\mapsto n-\si_1$,
\begin{multline} 
g_{n+1}
=\frac {1} {C
  \prod _{i=1} ^{s}(a_in-a_i\si_1+b_i)}
\sum_{\si=\si_1}^{\si_2}
\sum_{k=1}^{k_0}
\sum_{\ell=0}^k\binom k\ell
\sum_{0\le j_1,\dots,j_\ell<-\si_1}
f_{j_1}\cdots f_{j_\ell}
\\
\cdot
\underset{0\le j_{\ell+1},\dots,j_k\le n}{\sum_{j_{\ell+1}+\dots+j_k
    =n-\sigma-j_1-\dots-j_\ell+(k-\ell-1)\si_1}}
P_{\si, k}(n-\si_1,j_1,\dots,j_\ell,j_{\ell+1}-\si_1,
\dots,j_k-\si_1)g_{j_{\ell+1}}\cdots g_{j_k},\\
\quad \text{for }n\ge\bar N.
\label{eq:Rek-Polya2}
\end{multline}
Define $\bar\si:=\sigma+j_1+\dots+j_\ell-(k-\ell-1)\si_1$.
If $k>\ell+1$, we have
$\bar\si\ge \sigma-\si_1\ge0$.
This implies that in this case the inner sum in~\eqref{eq:Rek-Polya2}
runs over non-negative integers $j_{\ell+1},\dots, j_k$ with
$j_{\ell+1}+\dots+j_k=n-\bar\sigma$, where $\bar\si\ge0$.
The case $k=\ell+1$ does not need to be considered since then the inner sum
consists of just one (or no term).
The equality $k=\ell$ produces an empty sum since
the condition on the summation indices in the inner sum
in~\eqref{eq:Rek-Polya2} (the empty sum being equal to zero) becomes
\begin{equation} \label{eq:contra} 
0=n-\sigma-j_1-\dots-j_\ell+(k-\ell-1)\si_1
=n-\sigma-j_1-\dots-j_k-\si_1
\end{equation}
in this case, but
$$
n-\sigma-j_1-\dots-j_k-\si_1
> \bar N-\sigma+k\si_1-\si_1
\ge \si_2-\sigma-\si_1> 0
$$
by the definition of $\bar  N$, a contradiction with~\eqref{eq:contra}.

In summary, the recurrence \eqref{eq:Rek-Polya2} for
$\big(g_n\big)_{n\ge0}=\big(f_{n-\si_1}\big)_{n\ge0}$ is indeed a
recurrence of the form~\eqref{eq:Rek-Polya} without negative~$\si$'s;
in other words, it is of the form~\eqref{eq:Rek-Polya} with $\si_1=0$.
\end{proof}

As mentioned before, Sibuya and Sperber proved an important lemma that leads to the same conclusion as our Lemma \ref{lem:neg->pos}. Starting from a solution $f=\sum_{n=0}^\infty f_n x^n\in \Qbar[[x]]$\break of a non-trivial differential algebraic equation $Q(x,y,\ldots, y^{(k)})=0$ with coefficients\break in~$\Qbar$, they first consider a non-trivial differential equation $\widetilde{Q}(x,y,\ldots, y^{(\ell)})=0$ with $\widetilde{Q}\in \Qbar[X, Y_0, \ldots, Y_\ell]$ of which $f$ is still a solution. They require that $\ell\ge 1$ is minimal amongst all differential algebraic equations satisfied by $f$, and that the degree of $\widetilde{Q}$ in $Y_\ell$ is also minimal. Notice that $\ell$ is the transcendence degree of the field generated over~$\Qbar(x)$ by~$f$ and all its derivatives. 
This in particular ensures that  
$$
\frac{\partial \widetilde{Q}}{\partial Y_\ell}(x,f, \ldots,  f^{(\ell)})\neq 0,
$$ 
which is crucial. However, Sibuya and Sperber give no way to effectively compute~ $\ell$ and to minimize the degree in $Y_\ell$ starting from~$f$ and~$Q$ above. 
They canonically attach to $\widetilde{Q}$ and~$f$ a linear operator $L_0\in \Qbar[[x]][\frac{d}{dx}]$ of order $\ell$ and then consider the indicial polynomial $P_0$  at the origin of $L_0$. Then \cite[Lemma~2.2]{SiSpAA} says the following:

\medskip\noindent
    {\em For any $c\ge 0$, there exist integers $N'\ge 0,N\ge N', N''\ge c$ such that $u:=\sum_{n=N}^\infty f_n x^{n-N'} \in \Qbar[[x]]$ is a solution of a differential equation
$L(u)=x^{N''} F(x,u,u', \ldots, u^{(\ell)})$ where $F\in \mathcal{O}_{\Qbar}[X, Y_0, \ldots, Y_{\ell}]$ and $L$ is a linear differential operator in $\mathcal{O}_{\Qbar}[x][\frac{d}{dx}]$ of order $\ell$
and whose leading coefficient is independent of $N''$.}

\medskip
Moreover, the indicial polynomial $P(X)$ at the origin of $L$ is simply $P_0(X+N')$; in particular, if $P_0$ is split over $\mathbb Q$, then $P$ is as well. 
From the particular form of the equation satisfied by $u$, Sibuya and Sperber deduce a recurrence for $(f_n)_{n\ge 0}$ which is of the form \eqref{eq:Rek-Polyabis} with $\sigma_1\ge 0$ (provided $N''$ is large enough, which is possible because $c$ is arbitrary) and with $M(X)=P_0(X+N'+1)$. Therefore, this gives an interpretation of $M$ in terms of the indicial polynomial~$P_0$ attached to $\widetilde{Q}$ and $f$. It is important to observe that $P_0$ may be different for another solution in~$\Qbar[[x]]$ of the differential equation $Q(x,y,\ldots, y^{(k)})=0$.

Sibuya and Sperber also drew a few consequences from their lemma in \cite[p.~384]{SiSpAA} (encompassed by our Theorem~\ref{theo:main2}) when $P_0$ is a constant, when $0$ is an ordinary point of $L_0$, and when $F(x,u,u', \ldots, u^{(k)})=F(x)$ (which happens when $\widetilde{Q}$ is linear in $Y_0, \ldots, Y_k$). In the latter case, they recovered Popken's upper bound for the denominator of $f_n$ in the linear case; this is in fact equivalent to what Mahler had done in~\cite[p.~205]{mahler}.

\subsection{The main result}\label{ssec:mainresult}

We are now in the position to state our main result, Theorem~\ref{thm:main} below, which together with Lemma~\ref{lem:neg->pos} immediately implies Theorem~\ref{theo:main2}. Indeed, the denominator $d_n(C, D, \mathbf a,\mathbf b, k_0)$ in~\eqref{eq:dn} below obviously divides $\delta^{n+1}(\nu n+\nu)!^{2s}$ for all $n\ge N$ for suitable positive integers $\delta$ and $\nu$. We can also take $\delta$ large enough to ensure that $D$ divides~$\delta$. Hence, the denominator $d_n$ of $f_n$ divides $\delta^{n+1}(\nu n+\nu)!^{2s}$ for all $n\ge 0$. 

To prove Corollary \ref{coro:1}, let $v$ be a place of~$\Qbar$ over a rational prime number $p$. From the divisibility in Theorem \ref{theo:main2}, we deduce that 
$$
\vert f_n\vert_v \le p^{v_p(\delta^{n+1})+v_p((\nu n+\nu)!^{2s})} \le p^{v_p(\delta)(n+1)+\frac{2s\nu}{p-1}(n+1)},
$$
because by Legendre's formula \cite[p.~10]{LegeAA}
for the $p$-adic valuation of factorials, {\em i.e.}, 
\begin{equation} \label{eq:Leg} 
v_p(n!)=\sum_{\ell\ge1}\fl{\frac {n} {p^\ell}},
\end{equation}
we have
$$
v_p((\nu n+\nu)!^{2s}) \le 2s\sum_{\ell=1}^\infty\frac {\nu n+\nu} {p^\ell}=\frac{2s\nu(n+1)}{p-1}.
$$
This proves \eqref{eq:vadicbound}.

\medskip

The proof of Theorem~\ref{thm:main} requires a technical auxiliary result which is established separately in Proposition~\ref{prop:fastallg1} below. We emphasize that Theorem~\ref{thm:main} applies to  more general situations than the differential context behind Theorem~\ref{theo:main2}. Indeed, the only assumption we make below on the coefficients $P_{\si, k}(n,j_1,j_2,\dots,j_k)$ is that they are algebraic integers; they need not necessarily be piecewise polynomials in $n, j_1, j_2, \ldots, j_k$ as they are when they come from a solution of an algebraic differential equation.

\begin{theo} \label{thm:main}
Let us consider a recurrence of the form \eqref{eq:Rek-Polya} without
negative~$\si$'s, that is,
\begin{multline} \label{eq:Rek-Polya0} 
f_{n+1}
\\
=\frac {1} {C
  \prod _{i=1} ^{s}(a_in+b_i)}
\sum_{\si=0}^{\si_2}
\sum_{k=1}^{k_0}
\underset{0\le j_1,\dots,j_k\le n}{\sum_{j_1+\dots+j_k=n-\sigma}}
P_{\si, k}(n,j_1,j_2,\dots,j_k)f_{j_1}f_{j_2}\cdots
f_{j_k},\quad \text{for }n\ge N,
\end{multline}
where $P_{\si,k}(n,j_1,j_2,\dots,j_k)$ are algebraic integers, $\si_2$ is a non-negative integer, $C$, $s$ and $k_0$ are positive integers,
$\mathbf b=(b_1,b_2,\dots,b_s)$ is a vector of integers,
$\mathbf a=(a_1,a_2,\dots,a_s)$ is a vector of
positive integers such that $\gcd(a_i,b_i)=1$ for $i=1,2,\dots,s$, and $N$
is some non-negative integer such that
$a_iN+b_i-a_i\ge0$ for $i=1,2,\dots,s$.  The initial values $f_0,f_1,\dots,f_N$ are algebraic numbers with common denominator~$D\ge 1$.

Then, for all $n\ge N$, the denominator of~$f_n$ divides
\begin{equation} \label{eq:dn} 
d_n(C,D,\mathbf a,\mathbf b, k_0)
:=C^{n}D^{(k_0-1)n+1}\Pi^n
\prod _{i=1} ^{s}(a_in)!\,(a_in+b_i-a_i)!,
\end{equation}
where
\begin{equation} \label{eq:Pi} 
\Pi=
\prod _{i=1} ^{s}\Bigg(\big(\max\{b_i-a_i,0\}\big)!^{k_0-1}
\underset{p\text{ \em prime}}
{\prod_{p<2\max_{1\le j\le s}(b_j-a_j)}}
p^{\cl{\log_p\max\{b_i-a_i,1\}}}\Bigg).
\end{equation}
\end{theo}

\begin{proof}
The assertion is certainly true for $n=N$. For, by assumption,
the denominator of $f_N$ is $D$, which trivially
divides $d_N(C,D,\mathbf a,\mathbf b,k_0)$ due to the term
$D^{(k_0-1)N+1}$ in its definition.

For $n>N$ we are going to argue by induction.
Put
\begin{equation} \label{eq:phi-f} 
\ph_n=d_n(C,D,\mathbf a,\mathbf b,k_0)f_n=
C^{n}D^{(k_0-1)n+1}\Pi^n
\left(\prod _{i=1} ^{s}(a_in)!\,(a_in+b_i-a_i)!\right)f_n,
\quad \text{for }n> N.
\end{equation}
and
$$
\ph_n=
C^{n}D^{(k_0-1)n+1}\Pi^n
f_n,
\quad \text{for }0\le n\le N.
$$
By multiplying both sides of the recurrence~\eqref{eq:Rek-Polya0} by
$$
d_{n+1}(C,D,\mathbf a,\mathbf b,k_0)=
C^{n+1}D^{(k_0-1)(n+1)+1}\Pi^{n+1}
\prod _{i=1} ^{s}(a_in+a_i)!\,(a_in+b_i)!,
$$
we obtain
\begin{multline} \label{eq:Rek-Polya3} 
\ph_{n+1}=
\sum_{\si=0}^{\si_2}
\sum_{k=1}^{k_0}
C^\si D^{(k_0-1)\si+k_0-k}\Pi^{\si+1}
\underset{0\le j_1,\dots,j_k\le n}{\sum_{j_1+\dots+j_k=n-\sigma}}
\Bigg(\prod _{i=1} ^{s}
\frac {(a_in+a_i)!\,(a_in+b_i-1)!}
{\underset {j_t> N}{\prod\limits _{t=1} ^{k}}
  (a_ij_t)!\,(a_ij_t+b_i-a_i)!}
\Bigg)\\
\cdot
P_{\si, k}(n,j_1,j_2,\dots,j_k)
\ph_{j_1}\ph_{j_2}\cdots\ph_{j_k},
\quad \text{for }n\ge N,
\end{multline}
where (as indicated) the product over $t$ is taken only over those~$t$
between~1 and~$k$ for which $j_t> N$.

We claim that the recurrence \eqref{eq:Rek-Polya3} is one with
algebraic integer coefficients throughout. Obviously, by induction, this claim
would imply that all $\ph_n$'s --- including those with $n>N$ --- are algebraic 
integers. Going back to the definitions~\eqref{eq:phi-f} of $\ph_n$
and~\eqref{eq:dn} of $d_n(C,D,\mathbf a,\mathbf b,k_0)$, it
follows that the denominator of~$f_n$ divides $d_n(C,D,\mathbf
a,\mathbf b,k_0)$ for all $n\ge0$. 

Clearly, the only problem concerning our claim could
arise from the product over~$i$, a product of factorial ratios.
It is Proposition~\ref{prop:fastallg1} below which addresses this product.
In order to apply the proposition, we need to reorder the  $j_t$'s
in increasing order (so that the $\ta_i$ in the proposition are
well-defined). Then Proposition~\ref{prop:fastallg1} says that
the product of factorial ratios is ``almost" an integer, except
for some prime divisors~$p$ with $p<2\max_{1\le i\le s}(b_i-a_i)$.
However, their product is a divisor of~$\Pi$ (compare its definition
in~\eqref{eq:Pi} with the right-hand side of~\eqref{eq1:vpcoef2-k-s-si-ai}),
which in turn divides $\Pi^{\si+1}$ which appears
as a factor in~\eqref{eq:Rek-Polya3}. This proves our claim and
completes the proof of the theorem.
\end{proof}

We now proceed with the proof of the last missing step in the proof of Theorem~\ref{thm:main}. This is the most complicated part. 
\begin{prop} \label{prop:fastallg1}
Let $n,k,s$ and $a_1,a_2,\dots,a_s$ be positive integers
with $k\ge2$, let\break $b_1,b_2,\dots,b_s$ be integers, and let $\si$,
$\ta_1,\ta_2,\dots,\ta_{k}$, and
$j_1,j_2,\dots,j_{k}$ be
non-negative integers with $j_1+j_2+\dots+j_k=n-\si$ and such that
$a_ij_t+b_i-a_i<0$ for $1\le t\le \ta_i$ and $1\le i\le s$,
while $a_ij_t+b_i-a_i\ge0$ for $\ta_i+1\le t\le k$ and $1\le i\le s$,
In particular, if $b_i-a_i\ge0$, then necessarily $\ta_i=0$.

{\em(1)}
For prime numbers~$p\ge \max_{1\le i\le s}2(b_i-a_i)$,
we have
\begin{equation} \label{eq1:vpcoef1-k-s-si-ai} 
v_p\left(\prod _{i=1} ^{s}
\frac {(a_in+a_i)!\,(a_in+b_i-1)!}
      {(a_ij_{\ta_i+1})!\,(a_ij_{\ta_i+1}+b_i-a_i)!
  \cdots
  (a_ij_k)!\,(a_ij_k+b_i-a_i)!}\right)
    \ge2\sum_{i=1}^sv_p\big((a_i-1)!\,(a_i\si)!\big).
\end{equation}

{\em(2)}
For prime numbers~$p< \max_{1\le i\le s}2(b_i-a_i)$,
we have
\begin{multline} \label{eq1:vpcoef2-k-s-si-ai} 
v_p\left(\prod _{i=1} ^{s}
\frac {(a_in+a_i)!\,(a_in+b_i-1)!}
      {(a_ij_{\ta_i+1})!\,(a_ij_{\ta_i+1}+b_i-a_i)!
  \cdots
  (a_ij_k)!\,(a_ij_k+b_i-a_i)!}\right)\\
\ge
\sum_{i=1}^s\big(-\cl{\log_p\max\{b_i-a_i,1\}}
-(k-1)v_p\big(\!
\max\{b_i-a_i,0\}!\big)
+2v_p\big((a_i-1)!\,(a_i\si)!\big)
\big).
\end{multline}
\end{prop}

\begin{proof} For this proof, we drew inspiration from our previous work~\cite{kratriv}, in which we also proved another denominator conjecture in the context of zeta values, and not related to Mahler's denominator conjecture studied in the present paper.   

By Legendre's formula recalled above in \eqref{eq:Leg},  
we have
\begin{multline}
v_p\left(\prod _{i=1} ^{s}
\frac {(a_in+a_i)!\,(a_in+b_i-1)!}
      {(a_ij_{\ta_i+1})!\,(a_ij_{\ta_i+1}+b_i-a_i)!
  \cdots
  (a_ij_k)!\,(a_ij_k+b_i-a_i)!}\right)\\
=\sum_{\ell\ge1}\sum_{i=1}^s
\left(\fl{\frac {a_in+b_i-1} {p^\ell}}+\fl{\frac {a_in+a_i}
  {p^\ell}}\right.
\kern7.5cm
\\
\left.
-\fl{\frac {a_ij_{\ta_i+1}} {p^\ell}}-\dots-\fl{\frac {a_ij_k} {p^\ell}}
-\fl{\frac {a_ij_{\ta_i+1}+b_i-a_i} {p^\ell}}-\dots-\fl{\frac {a_ij_k+b_i-a_i} {p^\ell}}
\right).
\label{eq1:n,j,a_i-k-si}
\end{multline}
We put $N_i=\left\{\frac {a_in} {p^\ell}\right\}$,
$J_{t,i}=\left\{\frac {a_ij_t} {p^\ell}\right\}$, for $t
=1,2,\dots,k-1$ (sic!),
$A_i=\left\{\frac {a_i-1} {p^\ell}\right\}$ and
$B_i=\left\{\frac {b_i-a_i} {p^\ell}\right\}$ for $i=1,2,\dots,s$,
and $S_i=\left\{\frac {a_i\si} {p^\ell}\right\}$.
The reader should keep in mind that $N_i,J_{t,i},A_i,B_i,S_i$ also depend
on~$\ell$. We do not indicate this in the notation for better
readability. 

Using the above notation, the summand of the sum over~$\ell$
on the right-hand side
of~\eqref{eq1:n,j,a_i-k-si} can be rewritten as
{\allowdisplaybreaks
\begin{multline}
\sum_{i=1}^s\bigg(\fl{N_i+A_i+B_i}+\fl{N_i+A_i+\tfrac {1} {p^\ell}}
-\sum_{t=\ta_i+1}^{k-1}\big(\fl{J_{t,i}}+\fl{J_{t,i}+B_i}\big)\\
-\fl{N_i-J_{1,i}-\dots-J_{k-1,i}-S_i}-\fl{N_i-J_{1,i}-\dots-J_{k-1,i}+B_i-S_i}\\
+2\sum_{t=1}^{\ta_i}\fl{\tfrac {a_ij_t} {p^\ell}}+2\fl{\tfrac {a_i-1} {p^\ell}}
-(k-1-\ta_i)\fl{\tfrac {b_i-a_i} {p^\ell}}+2\fl{\tfrac {a_i\si} {p^\ell}}\bigg)\\
=\sum_{i=1}^s\bigg(\fl{N_i+A_i+B_i}+\fl{N_i+A_i+\tfrac {1} {p^\ell}}
-\sum_{t=\ta_i+1}^{k-1}\fl{J_{t,i}+B_i}
\kern4cm\\
-\fl{N_i-J_{1,i}-\dots-J_{k-1,i}-S_i}-\fl{N_i-J_{1,i}-\dots-J_{k-1,i}+B_i-S_i}\\
+2\sum_{t=1}^{\ta_i}\fl{\tfrac {a_ij_t} {p^\ell}}+2\fl{\tfrac {a_i-1} {p^\ell}}
-(k-1-\ta_i)\fl{\tfrac {b_i-a_i} {p^\ell}}+2\fl{\tfrac {a_i\si} {p^\ell}}\bigg).
\label{eq1:N,J,A_i-k-si}
\end{multline}}%

\medskip
(1)
Concerning the summand, we claim that
\begin{multline} \label{eq1:Teil1}
\fl{N_i+A_i+B_i}+\fl{N_i+A_i+\tfrac {1} {p^\ell}}
-\sum_{t=\ta_i+1}^{k-1}\fl{J_{t,i}+B_i}
\kern4cm\\
-\fl{N_i-J_{1,i}-\dots-J_{k-1,i}-S_i}-\fl{N_i-J_{1,i}-\dots-J_{k-1,i}+B_i-S_i}
\\
+2\sum_{t=1}^{\ta_i}\fl{\tfrac {a_ij_t} {p^\ell}}+2\fl{\tfrac {a_i-1} {p^\ell}}
-(k-1-\ta_i)\fl{\tfrac {b_i-a_i} {p^\ell}}+2\fl{\tfrac {a_i\si} {p^\ell}}\\
\ge
2\fl{\tfrac {a_i-1} {p^\ell}}
+2\fl{\tfrac {a_i\si} {p^\ell}}.
\end{multline}

For a fixed $i$, we consider first the case where $b_i<a_i$.
Then we have
\begin{equation} \label{eq:k-1-tau} 
\sum_{t=\ta_i+1}^{k-1}\fl{J_{t,i}+B_i}\le k-1-\ta_i
\le -(k-1-\ta_i)\fl{\tfrac {b_i-a_i} {p^\ell}}
\end{equation}
and
\begin{equation} \label{eq:Ni+Bi} 
\fl{N_i+A_i+B_i}\ge
\fl{N_i-J_{1,i}-\dots-J_{k-1,i}+B_i-S_i},
\end{equation}
which establishes \eqref{eq1:Teil1} in this case.

Now let $b_i\ge a_i$. We note that then necessarily $\ta_i=0$.
Without loss of generality  we assume that, for fixed~$i$
with $1\le i\le s$,
we have $J_{t,i}+B_i\ge1$
for $1\le t\le m_i$ and $J_{t,i}+B_i<1$ for $m_i+1\le t\le k-1$,
for some non-negative integer~$m_i$.

Since by assumption $p\ge2(b_i-a_i)$,
we have $\fl{\frac {b_i-a_i} {p^\ell}}=0$ and
$$
B_i=\left\{\frac {b_i-a_i} {p^\ell}\right\}
=\frac {b_i-a_i} {p^\ell}\le \frac {b_i-a_i} {2(b_i-a_i)}=\frac
{1} {2}. 
$$

It should be observed that the same conclusions also hold for
$b_i-a_i=0$. 

The case of $m_i=0$ of our claim follows from~\eqref{eq:Ni+Bi}.

Let now $m_i\ge1$.
If $N_i+B_i<1$, then 
$$
N_i+B_i-J_{1,i}-\dots-J_{k-1,i}-S_i<1-\tfrac {m_i} {2}.
$$
This implies
\begin{equation*} 
\fl{N_i+B_i-J_{1,i}-\dots-J_{k-1,i}-S_i}\le-\cl{\tfrac {m_i-1} {2}}.
\end{equation*}
Since in the current case $N_i+B_i<1\le J_{1,i}+B_i$,
we obtain that $N_i<J_{1,i}$. Hence
$$
N_i-J_{1,i}-\dots-J_{k-1,i}-S_i<
-J_{2,i}-\dots-J_{k-1,i}-S_i\le-\tfrac {m_i-1} {2},
$$
and therefore
\begin{equation*} 
  \fl{N_i-J_{1,i}-\dots-J_{k-1,i}-S_i}\le-\cl{\tfrac {m_i
    } {2}}.
\end{equation*}
If one uses these findings in \eqref{eq1:Teil1}, the left-hand side
can be bounded from below by
$$
-m_i+\cl{\tfrac {m_i-1} {2}}+\cl{\tfrac {m_i} {2}}
+2\fl{\tfrac {a_i-1} {p^\ell}}
+2\fl{\tfrac {a_i\si} {p^\ell}}
=2\fl{\tfrac {a_i-1} {p^\ell}}
+2\fl{\tfrac {a_i\si} {p^\ell}},
$$
as we claimed.

Now let $N_i+B_i\ge1$. Then
\begin{equation*} 
N_i+B_i-J_{1,i}-\dots-J_{k-1,i}-S_i<\tfrac {3} {2}-\tfrac {m_i} {2},
\end{equation*}
which implies
\begin{equation*} 
\fl{N_i+B_i-J_{1,i}-\dots-J_{k-1,i}-S_i}\le1-\cl{\tfrac {m_i} {2}}.
\end{equation*}
On the other hand, we have
$$
N_i-J_{1,i}-\dots-J_{k-1,i}-S_i<1-\tfrac {m_i} {2},
$$
and hence
\begin{equation*} 
\fl{N_i-J_{1,i}-\dots-J_{k-1,i}-S_i}\le-\cl{\tfrac {m_i-1} {2}}.
\end{equation*}
If one uses these findings in \eqref{eq1:Teil1}, the left-hand side
can be bounded from below by
$$
1-m_i-1+\cl{\tfrac {m_i} {2}}+\cl{\tfrac {m_i-1} {2}}
+2\fl{\tfrac {a_i-1} {p^\ell}}
+2\fl{\tfrac {a_i\si} {p^\ell}}
=2\fl{\tfrac {a_i-1} {p^\ell}}
+2\fl{\tfrac {a_i\si} {p^\ell}},
$$
again as we claimed.

\medskip
(2) Let $p< \max_{1\le i\le s}2(b_i-a_i)$.
As in Case~(1), if $b_i-a_i<0$ for some~$i$, then, since
$\ta_i\le k-1$, we have again
\eqref{eq:k-1-tau} and~\eqref{eq:Ni+Bi}, and then the
subsequent arguments show that~\eqref{eq1:Teil1} holds.

If $b_i-a_i>0$, consider an $\ell$ with $\ell\ge\log_p(b_i-a_i)+1$.
We then have $\fl{\frac {b_i-a_i} {p^\ell}}=0$
and
$$
B_i=\left\{\frac {b_i-a_i} {p^\ell}\right\}
=\frac {b_i-a_i} {p^\ell}\le \frac {b_i-a_i} {(b_i-a_i)p}=\frac {1} {p}\le\frac
{1} {2}. 
$$
Again, it should be observed that the same conclusions also hold for
$b_i-a_i=0$, even regardless of the value of~$\ell$. 
The arguments of Case (1) then show that \eqref{eq1:Teil1} holds again.

On the other hand, if $\ell<\log_p(b_i-a_i)+1$, then there is no
restriction on~$B_i$ (except the trivial strict upper bound of~1).
We claim that the relevant summand in~\eqref{eq1:N,J,A_i-k-si}
(in particular, with $\ta_i=0$) satisfies the estimation
\begin{multline} \label{eq1:Teil2}
\fl{N_i+A_i+B_i}+\fl{N_i+A_i+\tfrac {1} {p^\ell}}
-\sum_{t=1}^{k-1}\fl{J_{t,i}+B_i}
\kern4cm\\
-\fl{N_i-J_{1,i}-\dots-J_{k-1,i}-S_i}-\fl{N_i-J_{1,i}-\dots-J_{k-1,i}+B_i-S_i}
\\
+2\fl{\tfrac {a_i-1} {p^\ell}}
-(k-1)\fl{\tfrac {b_i-a_i} {p^\ell}}+2\fl{\tfrac {a_i\si} {p^\ell}}\\
\ge
2\fl{\tfrac {a_i-1} {p^\ell}}
-(k-1)\fl{\tfrac {b_i-a_i} {p^\ell}}
+2\fl{\tfrac {a_i\si} {p^\ell}}-1.
\end{multline}

Using the same notation as before, if $m_i=0$ then one sees
that~\eqref{eq1:Teil2} holds.

If $N_i+B_i<1$, then we have
$$
N_i+B_i-J_{1,i}-\dots-J_{k-1,i}-S_i<1-\tfrac {m_i} {2}.
$$
This implies
\begin{equation*} 
\fl{N_i+B_i-J_{1,i}-\dots-J_{k-1,i}-S_i}\le-\cl{\tfrac {m_i-1} {2}}.
\end{equation*}
Since in the current case $N_i+B_i<1\le J_{1,i}+B_i$,
we obtain that $N_i<J_{1,i}$. Hence
$$
N_i-J_{1,i}-\dots-J_{k-1,i}-S_i<
-J_{2,i}-\dots-J_{k-1,i}-S_i\le-\tfrac {m_i-1} {2},
$$
and therefore
\begin{equation*} 
  \fl{N_i-J_{1,i}-\dots-J_{k-1,i}-S_i}\le-\cl{\tfrac {m_i
    } {2}}.
\end{equation*}
If one uses these findings in \eqref{eq1:Teil2}, the left-hand side
can be bounded from below by
\begin{multline*}
-m_i+\cl{\tfrac {m_i-1} {2}}+\cl{\tfrac {m_i} {2}}
+2\fl{\tfrac {a_i-1} {p^\ell}}
-(k-1)\fl{\tfrac {b_i-a_i} {p^\ell}}
+2\fl{\tfrac {a_i\si} {p^\ell}}\\
=2\fl{\tfrac {a_i-1} {p^\ell}}
-(k-1)\fl{\tfrac {b_i-a_i} {p^\ell}}
+2\fl{\tfrac {a_i\si} {p^\ell}},
\end{multline*}
which is even slightly better than \eqref{eq1:Teil2}.

If $N_i+B_i\ge1$, then
\begin{equation*} 
N_i+B_i-J_{1,i}-\dots-J_{k-1,i}-S_i<{2}-\tfrac {m_i} {2},
\end{equation*}
which implies
\begin{equation*} 
\fl{N_i+B_i-J_{1,i}-\dots-J_{k-1,i}-S_i}\le1-\cl{\tfrac {m_i-1} {2}}.
\end{equation*}
Furthermore, we have
$$
N_i-J_{1,i}-\dots-J_{k-1,i}-S_i<1-\tfrac {m_i} {2},
$$
and hence
\begin{equation*} 
\fl{N_i-J_{1,i}-\dots-J_{k-1,i}-S_i}\le-\cl{\tfrac {m_i-1} {2}}.
\end{equation*}
If one uses these findings in \eqref{eq1:Teil2}, the left-hand side
can be bounded from below by
\begin{multline*}
1-m_i-1+2\cl{\tfrac {m_i-1} {2}}
+2\fl{\tfrac {a_i-1} {p^\ell}}
-(k-1)\fl{\tfrac {b_i-a_i} {p^\ell}}
+2\fl{\tfrac {a_i\si} {p^\ell}}\\
\ge-1+2\fl{\tfrac {a_i-1} {p^\ell}}
-(k-1)\fl{\tfrac {b_i-a_i} {p^\ell}}
+2\fl{\tfrac {a_i\si} {p^\ell}},
\end{multline*}
as we claimed.

If we use all this in \eqref{eq1:n,j,a_i-k-si} for fixed~$i$
(with $b_i-a_i>0$), we obtain
\begin{align*}
  v_p&
  \left(
  \frac {(a_in+a_i)!\,(a_in+b_i-1)!}
        {(a_ij_{\ta_i+1})!\,(a_ij_{\ta_i+1}+b_i-a_i)!
  \cdots
  (a_ij_k)!\,(a_ij_k+b_i-a_i)!}\right)\\
&\ge \sum_{\ell=1}^{\cl{\log_p(b_i-a_i)}}
\left(-1+2\fl{\tfrac {a_i-1} {p^\ell}}
-(k-1)\fl{\tfrac {b_i-a_i}
  {p^\ell}}+2\fl{\tfrac {a_i\si} {p^\ell}}\right)\\
&\kern4cm
+\sum_{\ell=\cl{\log_p(b_i-a_i)}+1}^\infty
\sum_{i=1}^s
\left(2\fl{\tfrac {a_i-1} {p^\ell}}+2\fl{\tfrac {a_i\si} {p^\ell}}\right)\\
&\ge
-\cl{\log_p(b_i-a_i)}
-(k-1)v_p(\max\{b_i-a_i,0\}!)+2v_p\big((a_i-1)!\,
(a_i\si)!\big).
\end{align*}
Upon summation of both sides over $1\le i\le s$,
this yields the claimed result in Case~(2).

\medskip
This completes the proof of the proposition, and thus of Theorem~\ref{thm:main}.
\end{proof}

\section{Examples} \label{sec:examples}

In this section, we present various examples of classical DA functions to which our results apply; many more examples of algebraic differential equations can be found in~\cite{kamke}, for instance ``Abel's equation'' on page 199 that we do not treat here. Due to the general nature of Theorems~\ref{theo:main2} and~\ref{thm:main}, the divisibility properties they imply in such or such situation are not always sharp. Better divisibility properties have often been obtained by different means, in particular when the Taylor coefficients are also known to satisfy combinatorial properties from which alternative formulas equivalent to 
those of type~\eqref{eq:Rek-Polya}  can be deduced.

\subsection{The linear case} \label{ssec:linearcase}

The material in this subsection is well-known. We recall it for the reader's convenience and also because it provides a simple interpretation of the condition that $M$ is split over $\mathbb Q$ in terms of an indicial polynomial of a linear differential equation. In the general non-linear case, the interpretation of~$M$ along a similar line is given in~\S\ref{ssec:lemma1}. The situation we consider here is when $Q$ in~\eqref{eq:diffalgeq} is a linear polynomial in $Y_0, \ldots, Y_k$, {\em i.e.}, an inhomogeneous linear equation for $f$. Differentiating enough time, we then obtain a homogeneous differential equation for $f$.

A linear differential equation $L:\sum_{k=0}^n p_k(x)f^{(k)}=0$, $p_0, \ldots, p_n\in \mathbb K[X]$, $p_n\neq 0$ and $\mathbb K$ any subfield of $\mathbb C$, can be rewritten as $\sum_{j=0}^m x^{m-j}q_j(\theta)(f)=0$, where $q_0, \ldots, q_m\in \mathbb K[X]$, $q_m\neq 0$, and $\theta$ is Euler's operator $x\frac{d}{dx}$. The polynomial $q_m(X)$ is the indicial polynomial at the origin of $L$; see \cite[\S4.1, Lemma~1]{firi}, the proof of which holds in fact {\em verbatim} in this general setting. Let us assume that $f(x)=\sum_{n=0}^\infty f_n x^n\in \mathbb K[[x]]$ is a solution of $L$. It is immediate that (see \cite[\S4.1]{batstrat})
$$
\sum_{j=0}^m q_j(n+j)f_{n+j}=0
$$
for all $n$ large enough, say $n\ge N$. Up to increasing the value of $N$, we can assume that for all $n\ge N$ we have $q_m(n+m)\neq 0$. Then
$$
f_{n+m}=-\frac{1}{q_m(n+m)}\sum_{j=0}^{m-1} q_j(n+j)f_{n+j}, \quad n\ge N.
$$

Let us now assume that $\mathbb K=\Qbar$. We may rewrite the above equation as 
$$
f_{n+1}:=\frac{1}{a_m(n)}\sum_{j=0}^{m-1} a_j(n)f_{n-j}, \quad n\ge N+m-1,
$$
where $a_0, \ldots, a_m\in \mathcal{O}_{\Qbar}[X]$ and the leading coefficient of 
$a_m$ is a positive integer. This recurrence is of the type in~\eqref{eq:Rek-Polya} with $k_0=1$. By induction on $n$, we obtain that  
$$
D_n:=D\prod_{j=N+m-1}^{n-1} a_m(j)
$$ 
is a denominator of $f_n$ for all $n\ge N+m$, where $D$ is a common denominator of $f_N, \ldots,\break f_{N+m-1}$.

Now, $a_m$ is split over $\mathbb Q$ if and only if the indicial polynomial at the origin $q_m$ is split over~$\mathbb Q$. Assuming this is the case, we have $a_m(X)=C\prod_{i=1}^d(u_iX+v_i)$ with $C$ a non-negative integer, $u_1, \ldots, u_d$ positive integers and $v_1, \ldots, v_d$
integers. It follows that $D_n$ divides $D C^{n+1}(\nu n+\nu)!^d$ for some positive integer $\nu$. 

\subsection{Roots of power series}\label{ssec:roots}

Let $f(x)=\sum_{n=0}^\infty f_n x^n\in \Qbar[[x]]$ be a solution of a linear equation with coefficients in~$\Qbar[x]$ and with rational exponents at the origin. We have seen in \S\ref{ssec:linearcase} that the sequence $(f_n)_{n\ge0}$ is then solution of a recurrence of type~\eqref{eq:Rek-Polya}, which in this case reads
\begin{equation}\label{eq:recurrenceroot1}
\sum_{k=0}^d p_k(n)f_{n+k}=0,
\end{equation}
with $p_k(X)\in \mathcal{O}_{\Qbar}[X]$ for $k=0, \ldots, d-1$, and $p_d(X)\in \mathbb Q[X]$ and split over $\mathbb Q$.

For simplicity, let us assume without loss of generality that $f_0\neq 0$. Let $m\ge 2$ be an integer and let $g(x)=\sum_{n=0}^\infty g_n x^n\in \Qbar[X]$ be such that $f=g^m$. In particular, $g_0\neq 0$ because $f_0=g_0^m$.  Substituting $g^m$ for $f$ in the differential equation satisfied by $f$ readily implies that $g$ is a DA function. We claim that the sequence $(g_n)_{n\ge 0}$ is also solution of a recurrence of type~\eqref{eq:Rek-Polya}.

Indeed, we have 
$$
f_n=\sum_{j_1+j_2+\cdots+j_m=n} g_{j_1}g_{j_2}\cdots g_{j_m} =
mg_0^{m-1}g_n +\sum_{\underset{j_1, \ldots, j_m\le n-1}{j_1+j_2+\cdots+j_m=n}} g_{j_1}g_{j_2}\cdots g_{j_m},
$$
so that from \eqref{eq:recurrenceroot1} we deduce that 
\begin{multline} \label{eq:recurrenceroot2}
mg_0^{m-1}p_d(n)g_{n+d}
\\
=-\sum_{k=0}^{d-1} \Big(p_k(n)\sum_{j_1+j_2+\cdots+j_m=n+k} g_{j_1}g_{j_2}\cdots g_{j_m}\Big)-p_d(n)\sum_{\underset{j_1, \ldots, j_m\le n+d-1}{j_1+j_2+\cdots+j_m=n+d}} g_{j_1}g_{j_2}\cdots g_{j_m}.
\end{multline}
Since $p_d(n)\neq 0$ for all $n$ large enough, it follows that \eqref{eq:recurrenceroot2} can be rearranged into a recurrence of type~\eqref{eq:Rek-Polya} as claimed.

\subsection{Generating functions of special sequences} 
We now illustrate our method with three classical series solutions of Riccati equations.

\medskip
The function $x/\log(1+x)=\sum_{n=0}^\infty g_nx^n$ is solution of $x(1+x)y'+y^2-(1+x)y=0$. This differential equation provides the recurrence relation 
\begin{equation} \label{eq:Gregory}
g_{n+1}=-\frac{1}{n+2}\Big((n-1)g_{n}+\sum_{j=1}^{n}g_j g_{n+1-j}\Big), \quad n\ge 0,
\end{equation}
with $g_0=1$. The $g_n$ are known as the {\it Gregory coefficients} or {\it Cauchy numbers of the first kind}. It is known that $n!\,g_n=\sum_{k=0}^n s(n,k)/(k+1)$, where $s(n,k)\in \mathbb Z$ are Stirling numbers of the first kind, so that  $\lcm\{1,2,\ldots, n+1\}n!\,g_n\in \mathbb Z$ for all $n\ge 0$. See \cite[pp.~293--294]{Comtet}.

Let us explain what divisibility could be obtained with our result. The recurrence~\eqref{eq:Gregory} is an example for which $\sigma_1=-1$, and thus this necessitates to transform it into a recurrence with $\sigma\ge 0$. Following the procedure in the proof of Lemma~\ref{lem:neg->pos}, 
let $\widetilde g_n=g_{n+1}$, substitute this in~\eqref{eq:Gregory}, and finally replace
$n$ by $n+1$. In this manner, we obtain
$$
\widetilde g_{n+1}=-\frac{1}{n+3}\Big(n\widetilde g_{n}+\sum_{j=0}^{n}\widetilde g_j \widetilde g_{n-j}\Big), \quad n\ge 0, \; \widetilde g_0=1/2.
$$
We can apply Theorem~\ref{thm:main} with $s=1$, $a_1=1$, $b_1=3$, $C=1$, $D=2$, $\sigma_1=\sigma_2=0$, $k_0=2$, $N=0$ and $\Pi=12$, so that $24^{n} (n-1)!\,(n+1)!\,{g}_{n} \in \mathbb Z$ for all $n\ge 1$.

\medskip

Next we consider the function $\frac{x}{e^x-1}=\sum_{n=0}^\infty b_n x^n$. We have $b_0=0, b_1=-1/2$, $b_{2n+1}=0$ for $n\ge 2$. Let us write $b_{2n}=B_{2n}/(2n)!$ for $n\ge 1$. The numbers $B_{2n}$ are the {\it Bernoulli numbers}. The Clausen--von Staudt Theorem says that the denominator of $B_{2n}$ is the product of the prime numbers $p$ such that $p-1$ divides $2n$; this quantity divides $\lcm\{1,2,\ldots, 2n+1\}$.

The above power series is solution of $xy'+y^2+(x-1)y=0$, from which we deduce that 
\begin{equation} \label{eq:bernoulli}
b_{n+1}=-\frac{1}{n+2}\Big(b_{n}+\sum_{j=1}^{n}b_j b_{n+1-j}\Big), \quad n\ge 0, \;b_0=1.
\end{equation}
This is again an example of a recurrence with $\sigma_1=-1$, and it is completely similar to~\eqref{eq:Gregory}. We let $\widetilde{b}_n:=b_{n+1}$, substitute in \eqref{eq:bernoulli}, and then change $n$ to $n+1$. We obtain
$$
\widetilde{b}_{n+1}=-\frac{1}{n+3}
\Big(\widetilde{b}_{n}+\sum_{j=0}^{n}\widetilde{b}_j \widetilde{b}_{n-j}\Big), \quad n\ge 0, \; \widetilde{b}_0=1/2.
$$
Exactly as with the Gregory coefficients, we can apply Theorem~\ref{thm:main} with $s=1$, $a_1=1$, $b_1=3$, $C=1$, $D=2$, $\sigma_1=\sigma_2=0$, $k_0=2$, $N=0$ and $\Pi=12$, so that $24^{n} (n-1)!\,(n+1)!\,{b}_{n} \in \mathbb Z$ for all $n\ge 1$.

\medskip

The function $\tan(x)=\sum_{n=0}^\infty t_n x^n$ is solution of $y'-y^2-1=0$, which translates into the recurrence
$$
t_{n+1}=\frac{1}{n+1}\sum_{j=0}^n t_jt_{n-j}, \; n\ge 1,
$$
with $t_0=0, t_1=1$. We can apply Theorem~\ref{thm:main} with $C=1$, $D=1$, $s=1$, $a_1=1$, $b_1=1$, $\sigma_1=\sigma_2=0$, $k_0=2$, $P_{0,1}=0$, $P_{0,2}=1$,  $N=1$ and $\Pi=1$. It follows that  $n!^2t_n \in \mathbb Z, n\ge 0$. It is known that $t_{2n}=0$ while $t_{2n-1}=(-1)^{n-1} 4^n(4^n-1)B_{2n}/(2n)!$, where the $B_{2n}$ are again the Bernoulli numbers. See \cite[p.~88]{Comtet}.

\medskip

These examples show why Theorem \ref{thm:main} is not always sharp: quantities like $\lcm\{1,2,\ldots, n\}$ or $\prod_{p \,:\, p-1 \vert n} p$ are  ``incorporated'' into an $n!$ factor, because even though they are less than~$3^n$, they
do not divide $\delta^n$ for some $\delta$. At the cost of a complicated $p$-adic analysis that we do not reproduce here, we have been able to prove that for the sequence $(f_n)_{n\ge 0}$ defined in the introduction by \eqref{eq:fna} with $a=1$, the denominator of $f_n$ divides $(n+1)!^2/\lcm\{1, 2, \ldots, n+1\}$, which improves on $n!\,(n+1)!$ given by Theorem~\ref{thm:main}. We did not try to prove such refinements in the general case.

\subsection{Elliptic differential equation} \label{ssec:elliptic}
The {\it Weierstra{\ss} elliptic function} $\wp$ with modular invariants $g_2,g_3\in \mathbb C$ (such that $g_2^3\neq 27g_3^2$) satisfies the algebraic differential equation
\begin{equation}\label{eq:equadiffelliptic}
\wp'^2=4\wp^3-g_2\wp-g_3.
\end{equation}
The function $\wp(x)$ is a doubly-periodic meromorphic function on $\mathbb C$, and it can be expanded in a Laurent series at $0$~(\footnote{Our results hold for DA Laurent series $f$, because $x^mf$ is a DA power series for some suitable integer~$m$. A non-linear recurrence of type~\eqref{eq:Rek-Polya} for the coefficients of~$x^mf$ obviously provides one for those of~$f$.}). We have 
$$
\wp(x)=\frac{1}{x^2}+\sum_{n=2}^\infty p_n x^{2n-2},
$$
where the coefficients $p_n$ satisfy the recurrence
\begin{equation}\label{eq:coeffelliptic}
p_n=\frac{3}{(2n+1)(n-3)} \sum_{j=2}^{n-2}p_jp_{n-j}, \quad n\ge 4,
\end{equation}
with $p_2=g_2/20, p_3=g_3/28$; see \cite[p.~635, Eq.~(18.5.3)]{abrastegun}. Eq.~\eqref{eq:coeffelliptic} is easily obtained not from~\eqref{eq:equadiffelliptic} but from the simpler equation
$$
12\wp^2-2\wp''-g_2=0,
$$
which follows at once from differentiating \eqref{eq:equadiffelliptic}.  
To apply Theorem~\ref{thm:main}, we define $u_n:=p_{n+2}$ for $n\ge 0$: we have
$$
u_{n+1}=\frac{3}{(2n+7)n}\sum_{j=0}^{n-1}u_j u_{n-j-1}, \quad n\ge 1,
$$
with $u_0=g_2/20$ and $u_1=g_3/28$. It is obvious that $u_n$ is a polynomial in $g_2$ and $g_3$ with rational coefficients. Assuming that $g_2, g_3$ are in~$\Qbar$, let $D$ denote a common denominator of $g_2/20$ and $g_3/28$. We are in the situation of Theorem~\ref{thm:main} with $C=1$, $s=2$, $a_1=1$, $b_1=0$, $a_2=2$, $b_2=7$, $k_0=2$, $\sigma_2=1$, $N=1$, $P_{0,1}=P_{0,2}=P_{1,0}=0$ and $P_{1,1}=3$ and  $\Pi:=2^3\cdot 3^2\cdot 5\cdot 7=2520$. It follows that 
$$
D^{n+1}\Pi^n (n-1)!\,n!\,(2n)!\,(2n+5)!\, p_{n+2} \in \mathcal{O}_{\Qbar}, \quad n\ge 1.
$$

We have not been able to find a divisibility of this nature in this generality in the literature, but would be surprised if it were a new result. For instance, Hurwitz obtained a much better result when $g_2=4$ and $g_3=0$ (the lemniscate case). Using his notations, let 
$$
\wp(x)=\frac{1}{x^2}+\sum_{n=1}^\infty \frac{2^{4n}E_n}{4n}\cdot\frac{x^{4n-2}}{(4n-2)!} \in \mathbb Q[[x]],
$$
which  is a solution of $y'^2=4y^3-4y$. 
Then he showed in \cite[p.~208]{hurwitz} the following result, which is an analogue of the Clausen--von Staudt Theorem for the Bernoulli numbers.

\medskip\noindent
{\em Let $d_n$ denote the denominator of $E_n$: {\em(1)} there is no prime number $\equiv 3 \mod 4$ that divides~$d_n$; {\em(2)} if a prime number $p\equiv 1 \mod 4$  divides $d_n$, then $p-1$ divides $4n$ and $p^2$ does not divide $d_n$.}

\medskip\noindent
Moreover, it follows from the partial fraction expansion of $E_n$ in \cite[p.~219, Eq.~(78)]{hurwitz} that $v_2(E_n)=-1$. Hence,
$(\prod_{p \,:\, p-1\vert 4n}p) E_n\in \mathbb Z$. 
Moreover,  letting
$$\frac1{2\wp(x)}=\sum_{n=0}^{\infty} e_{n} \frac {x^{4n+2}} {(4n+2)!},
$$
he also proved that  $e_{n}\in \mathbb Z$ for all $n\ge 0$; see \cite[p.~224]{hurwitz}. Finally, let us mention that $-1/\wp(x)$ is curiously also a solution of $y'^2=4y^3-4y$, and that $\sqrt{1/\wp(x)}$ is a solution of $y''+2y^3=0$.

\subsection{Painlev\'e equations}

Painlev\'e equations produce a lot of interesting examples.
(We refer the reader to~\cite{ClarAA} for an introduction.)
In this subsection, we only consider  the family of algebraic differential equations known as Painlev\'e \textrm{PII'}:
$$
y''=\delta(2y^3-2xy)+\gamma(6y^2+x)+\beta y + \alpha,
$$
where it is assumed that $\alpha, \beta, \gamma, \delta$ are in~$\Qbar$. It is a simple task to check that $u(x):=\sum_{n=0}^\infty u_n x^n$ is a solution of \textrm{PII'}, where the sequence $(u_n)_{n\ge 0}$ is given by 
\begin{equation*}
u_{n+1}=\frac{1}{n(n+1)}\Big(
2\delta\sum_{\underset{0\le j_1, j_2, j_3\le n-1}{j_1+j_2+j_3=n-1}}u_{j_1}u_{j_2}u_{j_3}+6\gamma\sum_{\underset{0\le j_1, j_2\le n-1}{j_1+j_2=n-1}}u_{j_1}u_{j_2}
-2\beta u_{n-1} -2\delta u_{n-2}\Big), \; n\ge 2,
\end{equation*}
with $u_0$, $u_1$ arbitrary in $\Qbar$, and $u_2=-\delta u_0^3-3\gamma u_0^2-\beta u_0/2-\alpha/2$. This is a situation where Theorem~\ref{thm:main} with $s=2$, $a_1=a_2=b_1=1$, $b_2=0$, $k_0=3$, $\sigma_1=1$, $\sigma_2=2$, $N=2$, $\Pi=1$, and $C$ the common denominator of $2\delta, 6\gamma, 2\beta$, applies. 
Letting $D\ge 1$ denote the least common denominator of $u_0, u_1, u_2$, we have 
$$
C^n D^{2n+1}(n-1)!\,n!^3\, u_n \in \mathbb Z, \quad n\ge 2.
$$

\subsection{Kepler's equation} \label{ssec:kepler}

{\it Kepler's equation} from Celestial Mechanics is 
$
E-e\sin(E)=M,
$
where $E, M, e$ are the eccentric anomaly, the mean anomaly, and the eccentricity, respectively. Inventing for the purpose of solving this equation what is now known as {\it Lagrange's inversion formula}, Lagrange~\cite{lagrange} found the well-known expansion
\begin{equation}\label{eq:lagrange}
\varphi(e):=E-M=\sum_{n=1}^\infty 
\Big(\frac{d}{dt}\Big)^{n-1}
\big(\sin(t)^n\big)_{t=M} \cdot \frac{e^n}{n!},
\end{equation}
which is an odd and $2\pi$-periodic function of $M$, and which converges in the disk $\vert e\vert<e_0$ for all $M$. Here, $e_0=2\sqrt{\rho(1-\rho)}/(1-2\rho)\approx 0.6627$, where $\rho\approx 0.0832$ satisfies ${1/\rho-1}=\exp(2/(1-2\rho))$; see the details in~\cite{puiseux}. 

It is easy to prove  that 
$
(\sin^n(t))^{(n-1)} =P_n\big(\sin(t),\cos(t)\big),
$
where $P_n(X,Y)\in \mathbb Z[X,Y]$ is of degree $\le n$ in $X$ and $\le 1$ in $Y$. Therefore, when $\sin(M)\in \Qbar$ (hence $\cos(M) \in \Qbar$ as well), there exists a positive integer $\delta$ such that, for all $n\ge 1$, 
$$
\delta^{n+1}\Big(\frac{d}{dt}\Big)^{n-1}\big(\sin^n(t)\big)_{t=M} \in \mathcal{O}_{\Qbar}.
$$
Hence, from \eqref{eq:lagrange}, we deduce in this case that the denominator of the $n$-th Taylor coefficient of $\varphi$ at the origin divides $\delta^{n+1}n!$.

\medskip

 We now want to use this example to show how to effectively determine a recurrence like~\eqref{eq:Rek-Polya} in a more complex situation than in the previous examples. We change $e$ to $x$ and let $\varphi(x)=E-M$. Kepler's equation reads $\sin(\varphi+M)=\varphi/x$, with $\varphi(0)=0$ and $\varphi'(0)=\sin(M)$. Differentiating this equation with respect to $x$ and squaring both sides, we find 
\begin{equation}\label{eq:varphidiff}
\varphi'^2(1-(\varphi/x)^2)=((\varphi/x)')^2,
\end{equation}
or, in expanded form,
\begin{equation}\label{eq:keplerdiff}
x^2(x^2-1)\varphi'^2-\varphi^2+2x\varphi'\varphi-x^2\varphi'^2\varphi^2=0.
\end{equation}
Let $\varphi(x)=\sum_{n=0}^\infty \phi_n x^n$. We have $\phi_0=0$, $\phi_1=\sin(M)$, and $\phi_2=\sin(M)\cos(M)$. We shall assume below that $\phi_2\neq 0$, which {\em a fortiori} implies that $\phi_1\neq 0$. From the algebraic differential equation~\eqref{eq:keplerdiff}, we deduce that, for $n\ge 5$, we have
\begin{multline*}
\sum_{\underset{0\le j_1, j_2\le n-4}{j_1+j_2=n-4}} (j_1+1)(j_2+1)\phi_{j_1+1}\phi_{j_2+1}-\sum_{\underset{0\le j_1, j_2\le n-2}{j_1+j_2=n-2}} (j_1+1)(j_2+1)\phi_{j_1+1}\phi_{j_2+1}
-\sum_{\underset{0\le j_1, j_2\le n}{j_1+j_2=n}} \phi_{j_1}\phi_{j_2}
\\
+\sum_{\underset{0\le j_1, j_2\le n-1}{j_1+j_2=n-1}} 2(j_1+1)\phi_{j_1+1}\phi_{j_2}-\sum_{\underset{0\le j_1, j_2, j_3, j_4\le n-2}{j_1+j_2+j_3+j_4=n-2}} (j_1+1)(j_2+1)\phi_{j_1+1}\phi_{j_2+1}\phi_{j_3}\phi_{j_4}=0,
\end{multline*}
which can be rewritten as
\begin{multline*}
\sum_{\underset{0\le j_1, j_2\le n-3}{j_1+j_2=n-2}} j_1j_2\phi_{j_1}\phi_{j_2}-\sum_{\underset{0\le j_1, j_2\le n-1}{j_1+j_2=n}} j_1j_2\phi_{j_1}\phi_{j_2}
-\sum_{\underset{0\le j_1, j_2\le n}{j_1+j_2=n}} \phi_{j_1}\phi_{j_2}
\\
+\sum_{\underset{0\le j_1\le n, 0\le j_2\le n-1}{j_1+j_2=n}} 2j_1\phi_{j_1}\phi_{j_2}-\sum_{\underset{0\le j_1, j_2\le n-1, 0\le j_3, j_4\le n-2}{j_1+j_2+j_3+j_4=n}} j_1j_2\phi_{j_1}\phi_{j_2}\phi_{j_3}\phi_{j_4}=0.
\end{multline*}
This is equivalent to 
\begin{multline*}
\sum_{\underset{0\le j_1, j_2\le n}{j_1+j_2=n-2}}P_1(n,j_1,j_2)\phi_{j_1}\phi_{j_2}+\sum_{\underset{0\le j_1, j_2\le n}{j_1+j_2=n}} P_2(n,j_1,j_2)\phi_{j_1}\phi_{j_2}
\\
+\sum_{\underset{0\le j_1, j_2,j_3,j_4\le n}{j_1+j_2+j_3+j_4=n}} P_3(n,j_1,j_2,j_3,j_4)\phi_{j_1}\phi_{j_2}\phi_{j_3}\phi_{j_4}=0,
\end{multline*}
where 
$$
P_1(n,j_1,j_2)=j_1j_2{\bf 1}_{\{0\le j_1,j_2\le n-3\}},
$$ 
$$P_2(n,j_1,j_2)=-j_1j_2{\bf 1}_{\{0\le j_1,j_2\le n-1\}}
+2j_1{\bf 1}_{\{0\le j_2\le n-1\}}-1,
$$
$$P_3(n,j_1,j_2,j_3,j_4)=-j_1j_2{\bf 1}_{\{0\le j_1,j_2\le n-1, 0\le j_3, j_4\le n-2\}}.
$$
This is an algebraic relation between $\phi_0, \phi_1, \ldots, \phi_n$. A simple analysis shows that neither $\phi_{n}$ nor $\phi_{n-1}$ appear in this relation. Indeed, in the three sums, $\phi_n$ can appear only in terms involving $\phi_n\phi_0$ or $\phi_n\phi_0^3$, which are equal to 0 because $\phi_0=0$. The coefficient $\phi_{n-1}$ does not appear in the first sum, appears in the second as $(P_2(n,n-1,1)+P_2(n,1,n-1))\phi_1\phi_{n-1}=0$ because $P_2(n,n-1,1)+P_2(n,1,n-1)=0$, and in the third sum in terms of the form $\phi_{n-1}\phi_1\phi_0\phi_0=0$. The coefficient $\phi_{n-2}$ does not appear in the first sum, appears in terms involving $\phi_{n-2}\phi_1^2\phi_0=0$ and $\phi_{n-2}\phi_2\phi_0^2=0$ in the third sum, and in the second sum, $\phi_{n-2}$ appears in the expression $(P_2(n,n-2,2)+P_2(n,2,n-2))\phi_2\phi_{n-2}=-2(n+1)\phi_2\phi_{n-2}$. Therefore, since we have assumed that $\phi_2\neq 0$, we finally obtain the recurrence (after changing $n$ to $n+3$)
\begin{multline*}
\phi_{n+1}=\frac{1}{2\phi_2(n+4)} \bigg(  
\sum_{\underset{0\le j_1, j_2\le n}{j_1+j_2=n+1}}P_1(n+3,j_1,j_2)\phi_{j_1}\phi_{j_2}
+
\sum_{\underset{0\le j_1, j_2\le n}{j_1+j_2=n+3}}P_2(n+3,j_1,j_2)\phi_{j_1}\phi_{j_2}
\\+\sum_{\underset{0\le j_1, j_2,j_3,j_4\le n}{j_1+j_2+j_3+j_4=n+3}} P_3(n+3,j_1,j_2,j_3,j_4)\phi_{j_1}\phi_{j_2}\phi_{j_3}\phi_{j_4}\bigg), \quad n\ge 2.
\end{multline*}
This is a recurrence of the form~\eqref{eq:Rek-Polya} (with negative $\sigma_1$), to which we can apply Theorem~\ref{theo:main2} because $M(X)\in \mathbb Q[X]$ is linear. We do not provide the details of the explicit result, which is of the form $\widetilde{\delta}^{n+1}(\nu n+\nu)!^2$ for some integers $\widetilde{\delta}, \nu\ge 1$, because it is weaker than the already obtained divisibility of $\delta^{n+1}n!$ by the denominator of $\phi_n$.

Finally, instead of $\varphi$, we could consider $\psi(x):=\varphi(x)/x=\sum_{n=0}^\infty \phi_{n+1}x^n$. The differential equation~\eqref{eq:varphidiff} becomes
$
(x\psi)'^2(1-\psi^2)=(\psi')^2
$
or, in expanded form,
\begin{equation*}
(1-x)\psi'^2+2x\psi'\psi^3+x^2\psi'^2\psi^2-2x\psi'\psi+\psi^4-\psi^2=0.
\end{equation*}
This would lead to an {\em a priori} different recurrence for the sequence $(\phi_{n})_{n\ge 0}$.

\subsection{The compositional inverse of the dilogarithm}

The dilogarithm is $\textup{Li}_2(x):=\sum_{n=1}^\infty x^n/n^2$, with derivative $-\log(1-x)/x$, and it admits a power series $\ell$ as inverse for the composition:  
$$
{\ell}(x):=\sum_{n=0}^\infty \ell_n x^n=x-\frac14x^2+\frac{1}{72}x^3-\frac{1}{576}x^4-\frac{31}{86400}x^5- \frac{149}{1036800}x^6-\cdots.
$$ 
The function $\ell$ is DA, for it satisfies
\begin{equation}\label{eq:equadiffdilog}
\ell''\ell-\ell''\ell^2+\ell'^3+\ell'^2\ell-\ell'^2=0.
\end{equation}
Indeed, we have $1=(\textup{Li}_2(\ell))'=-\ell'\log(1-\ell)/\ell$, so that, by multiplying by $\ell$ and then differentiating both sides, we obtain
$$
\ell'=-\ell''\log(1-\ell)+\frac{\ell'^2}{1-\ell}=\frac{\ell''\ell}{\ell'}+\frac{\ell'^2}{1-\ell}, 
$$
which gives the above differential equation after some rearrangement. The situation is a little simpler than with Kepler's equation because~\eqref{eq:equadiffdilog} is autonomous, {\em i.e.}, its coefficients are independent of $x$. For all $n\ge 0$, we have  
\begin{align}
[x^n](\ell''\ell) &=\sum_{i+j=n} (i+1)(i+2)\ell_{i+2}\ell_j, \label{eq:coeff1}
\\
[x^n](-\ell''\ell^2) &=-\sum_{i+j+k=n} (i+1)(i+2)\ell_{i+2}\ell_jf_k, \notag
\\
[x^n](\ell'^3) &=\sum_{i+j+k=n} (i+1)(j+1)(k+1)\ell_{i+1}\ell_{j+1}\ell_{k+1}, \label{eq:coeff3}
\\
[x^n](\ell'^2\ell) &=\sum_{i+j+k=n} (i+1)(j+1)\ell_{i+1}\ell_{j+1}\ell_k ,
\notag 
\\
[x^n](-\ell'^2) &=-\sum_{i+j=n} (i+1)(j+1)\ell_{i+1}\ell_{j+1},
\label{eq:coeff5}
\end{align}
where in the summations all the indices run between 0 and $n$, and where as usual $[x^n](f)$ denotes the $n$-th Taylor coefficient of a power series $f$. Proceeding as in \S\ref{ssec:kepler}, we see that $\ell_{n+2}$ appears in none of these sums because it is always multiplied by $\ell_0=0$. The term $\ell_{n+1}$ appears only in~\eqref{eq:coeff1}, \eqref{eq:coeff3}, and in~\eqref{eq:coeff5}, with the coefficients $n(n+1)\ell_1$, $3(n+1)\ell_1^2$, and $-2(n+1)\ell_1$, respectively. This is $\ell_1(n+1)(n+3\ell_1-2)=(n+1)^2$ in total. Therefore,  we have 
\begin{multline*}
\ell_{n+1}=\frac{1}{(n+1)^2}\Big( 
-\sum_{\underset{i\le n-2, j\le n}{i+j=n}} (i+1)(i+2)\ell_{i+2}\ell_j +
\sum_{\underset{i\le n-2, j,k\le n}{i+j+k=n}} (i+1)(i+2)\ell_{i+2}\ell_j\ell_k
\\
-\sum_{\underset{i,j,k\le n-1}{i+j+k=n}} (i+1)(j+1)(k+1)\ell_{i+1}\ell_{j+1}\ell_{k+1}-\sum_{\underset{i,j\le n-1, k\le n}{i+j+k=n}} (i+1)(j+1)\ell_{i+1}\ell_{j+1}\ell_k 
\\
+\sum_{\underset{i,j\le n-1}{i+j=n}} (i+1)(j+1)\ell_{i+1}\ell_{j+1}
\Big), \quad n\ge 0.
\end{multline*}
It follows that, for all $n\ge 0$, $\delta^{n+1} (\nu n+\nu)!^4\ell_n\in \mathbb Z$ for some positive integers $\delta, \nu$. Numerical experiments suggest that the exponent $4$ could be replaced by $2$. This would be true if it could be proved for instance that $\big(\ell(n)\big)_{n\ge 0}$ satisfies a non-linear recurrence of type \eqref{eq:Rek-Polya} like above but with $(n+1)^2$ replaced by $an+b$ for suitable integers $a$ and $b$.

As a side remark, we point out that it is possible to determine the asymptotic behaviour of $\ell_n$. By Lagrange's inversion formula \cite[p.~732,~Theorem~A.2]{FlSeAA}, we have
$$
\ell_n=\frac1{2i\pi n}\int_{\mathcal{C}} \frac{dz}{\textup{Li}_2(z)^n}, \quad n\ge 0,
$$
where $\mathcal{C}$ is a circle centred at 0 and of sufficiently small radius. This circle can be deformed to a suitable Hankel-type contour to  which the saddle point method can be applied (see~\cite{FlSeAA} for similar computations). We obtain that 
\begin{equation}\label{eq:asympellnzeta2}
\ell_n\sim-\frac{\zeta(2)^{-n}}{n^2\log(n)^2}, \quad n\to +\infty.
\end{equation}
Hence, the radius of convergence of $\ell$ at $x=0$ is $\pi^2/6$, and $(\ell_n/\ell_{n+1})_{n\ge 0}$ is a sequence of rational numbers that converge to $\pi^2/6$, but not especially quickly. The proof of \eqref{eq:asympellnzeta2} is given in our note \cite{kratrivpolylog}.

\subsection{Examples related to modular forms}

Various automorphic functions, including certain classical modular forms, are known to satisfy algebraic differential equations of order at least 3 and of the form $Q(y)(y')^2=\{y;x\}$, where  $\{y;x\}:=y'''/y'-\frac{3}{2}(y''/y')^2$ is the Schwarzian derivative with respect to $x$, and $Q$ is a rational function with poles of order at most $2$. For instance, Mahler proved in~\cite{mahler2} that the {\it elliptic modular invariant} 
\begin{multline*}
F(q):=J(q^2) = q^{-2}\Big(1+240\sum_{n=1}^\infty n^3q^{2n}(1-q^{2n})^{-1}\Big)^3\prod_{n=1}^\infty (1-q^{2n})^{-24}
\\
=1/q^2+744+196884q^2+21493760q^4+864299970q^6+\cdots,
\end{multline*}
satisfies the order 3 equation
\begin{equation}\label{eq:modularinvariant}
F'''=\frac{3q^2 F''^2-4q F'F''-F'^2}{2q^2F'}
-F'^3\Big(\frac{4}{9F^2}+\frac{3}{8(F-12^3)^2}-\frac{23}{72F(F-12^3)}\Big),
\end{equation}
where the derivatives are taken with respect to $q$. He also proved that $F$ cannot satisfy an algebraic differential equation of order $\le 2$. Since the degree in $F'''$ of \eqref{eq:modularinvariant} is 1, this equation is of the standard form to which we can apply the procedure described in \cite[p.~382]{SiSpAA}. The indicial polynomial $P_0$ of the associated differential operator $L_0\in \Qbar[[x]][\frac{d}{dx}]$ is $P_0(X)=X^3 + 8X^2 - 10X + 64$, which has no rational root. Lemma 2.2 of \cite[p.~382]{SiSpAA} then provides a non-linear recurrence of type \eqref{eq:Rek-Polyabis} for the coefficients of $F$ with $M(X)=P_0(X+m)$ for some integer $m\ge 0$.
This recurrence is not of type  \eqref{eq:Rek-Polya}, and  our Theorems~\ref{theo:main2} and~\ref{thm:main} cannot be applied to it. We do not know if there exists a recurrence of type  \eqref{eq:Rek-Polya} for the coefficients of~$F$ (which are well-known to be integers), {\em i.e.},  with $M$ split over~$\mathbb Q$. Let us also mention that a vast algebraico-geometric generalization of modular forms, known as mirror maps, are shown to be DA functions in~\cite[Proposition~1]{zudilin}.

\subsection{A simple example with $M$ non split over $\mathbb Q$} 
We conclude with an example for which $M$ is non-split over $\mathbb Q$, and for which it is possible to make numerical experiments. The differential equation $x^2y''+(x-1)y'+y-xy^2=1$ has as solution the power series 
$\sum_{n=0}^{\infty} f_n x^n$ with 
$$
f_{n+1}=\frac{1}{n^2+1}\sum_{k=0}^n f_k f_{n-k}, \quad n\ge 0, \; f_0:=1. 
$$
Numerical experiments for values of $n$ up to 2000 suggest that the denominator $d_n$ of $f_n$ is such that $\log(d_n)/(n\log(n)^2)$ converges to a constant close to $0.566$. Moreover, 
$$
\frac{\log(d_{2n})}{2n\log(2n)}-\frac{\log(d_{n})}{n\log(n)}
$$ seems to converge to a constant close to 
$0.39$, which tends to confirm that $\log(d_{n})/(n\log(n))$ behaves more like $\log(n)$ than like $\log\log(n)$. 

Moreover, it seems that $\prod_{k=0}^{\lfloor n/2\rfloor} (k^2+1)$ divides $2^n d_n$ for all $n\ge 0$. If true, this conjecturally implies that there exist no integers $\delta, \nu,\mu,s\ge 0$ such that $d_n$ divides $\delta^{n+1}(\nu n+\mu)!^s$ for all $n\ge 0$. Indeed, on the contrary, this would imply the same type of divisibility for $\prod_{k=0}^{n} (k^2+1)$, and we have explained in the introduction why this is not true conjecturally. 
 
All these computations do not point towards a positive solution to Mahler's denominator conjecture, but these experiments should be conducted for more values of $n$ on this sequence and others in the non-split case to be more affirmative.

\section*{Acknowledgement}
We are grateful to the Institut Fourier for having supported
a research visit of the first author at the institute in
November 2024, during which this project, started
at the occasion	of another research visit at the institute in 2018, could be eventually finalised.

\end{document}